\documentclass[preprint,11pt]{elsarticle}
\usepackage{lineno,hyperref}
\modulolinenumbers[5]

\usepackage{color}
\usepackage{amsmath}
\usepackage{amsfonts,amsthm,amssymb}
\usepackage{amsfonts}
\usepackage{graphics}
\usepackage{pstricks}
\usepackage{graphicx}
\usepackage{cleveref}
\usepackage{extarrows,chngpage,array,float}
\usepackage{amssymb}
\usepackage{latexsym,bm}
\usepackage{mathrsfs}
\usepackage{algorithm}
\usepackage{caption}
\usepackage{comment}
\usepackage{float}
\usepackage{tikz}
\usepackage{multicol} 
\usepackage{algpseudocode}
\usepackage{booktabs}
\usepackage{longtable}
\newtheorem{theorem}{Theorem}[section]
\newtheorem{lemma}[theorem]{Lemma}
\newtheorem{definition}[theorem]{Definition}
\newtheorem{proposition}[theorem]{Proposition}
\newtheorem{corollary}[theorem]{Corollary}

\newtheorem{problem}[theorem]{Problem}

\newtheorem{example}[theorem]{Example}
\allowdisplaybreaks   

\numberwithin{equation}{section}
\journal{}

\begin{document}

\begin{frontmatter}

\title{Combinatorial interpretations of Tutte polynomials at the point $(2,-1)$}

\author{Yu Chen$^a$, \ \ Xian'an Jin$^{a,b}$, \ \ Tianlong Ma$^{c,}$\footnote{Corresponding author.} 
\\[1ex]
\small $^a$School of Mathematical Sciences\\[-0.8ex]
\small  Xiamen University\\[-0.8ex]
\small P. R. China	
\\[0.3ex]
\small $^b$Qinghai Minzu University\\[-0.8ex]
\small P. R. China	
\\[0.3ex]
\small $^c$School of Science\\[-0.8ex]
\small Jimei University\\[-0.8ex]
\small P. R. China
\\
\small\tt Email: yuch236546@outlook.com, xajin@xmu.edu.cn, tianlongma@aliyun.com}

\begin{abstract}

Let $G$ be a simple connected graph, and let $T_{G}(x,y)$ be the Tutte polynomial of $G$. Motivated by the works in \cite{Ma}, we, in this paper, introduce the even-left spanning forests of $G$ and odd $G$-partitionable permutations, and show that $T_{G}(2,-1)$ is equal to both the number of even-left spanning forests of $G$ and the number of odd $G$-partitionable permutations. In particular, for a complete graph $K_n$, we prove that $T_{K_{n}}(2,-1)$ is the number of alternating permutations on $\{1,2,\dots,n+1\}$, using two distinct techniques: a recurrence relation and an explicit bijection construction.

\end{abstract}


\end{frontmatter}

\section{Introduction}

The Tutte polynomial, defined by Tutte \cite{Tutte} in 1954, is a generalization of the chromatic polynomial. Given a  graph $G=(V,E)$, the Tutte polynomial of $G$ is 
$$T_{G}(x,y)=\sum_{A\subset E}(x-1)^{r(G)-r(A)}(y-1)^{\lvert A \rvert -r(A)},$$
where $r(A)$ is the rank of $A$, defined as $\lvert A\rvert -c(A)$, and $c(A)$ is the number of connected components of the spanning subgraph $(V,A)$. It is easy to see from the above definition of Tutte polynomial that $T_{G}(1,1)$, $T_{G}(2,1)$, $T_{G}(1,2)$ and $T_{G}(2,2)$ encode the number of spanning trees of $G$, spanning forests of $G$, connected spanning subgraphs of $G$ and spanning subgraphs of $G$ respectively. 

Combinatorial interpretations of Tutte polynomials at integer point have been studied extensively. In 1973, Stanley \cite{Stanley} proved that $T_{G}(2,0)$ is the number of acyclic orientations of $G$. In 1978, Read and Rosenstiehl \cite{Rosenstiehl} showed that $T_G(-1, -1) = (-1)^{|E(G)|}(-2)^{\dim(\mathcal{B})}$, where $\mathcal{B} = \mathcal{C} \cap \mathcal{C}^\perp$ and $\mathcal{C},\ \mathcal{C}^\perp$ are the cycle space, cocycle space of $G$ over the finite field $GF(2)$ respectively. In particular, if $G$ is a connected plane graph, Martin \cite{Martin} showed $T_G(-1, -1) = (-1)^{|E(G)|}(-2)^{a(\tilde{G}_m) - 1}$, where $\tilde{G}_m$ is the directed medial graph of $G$ and $a(\tilde{G}_m)$ is the number of anticircuits in $\tilde{G}_m$.
In 1983, Greene and Zaslavsky \cite{Greene} proved that $T_{G}(1,0)$ is the number of acyclic orientations of $G$ with a single specified source and $T_{G}(0,2)$ is the number of totally cyclic orientations of $G$. In 2007, Gioan \cite{Gioan} proved $T_{G}(0,1)$ is the number of score vectors of strongly connected orientations of $G$.

Let $G$ be a simple connected graph. In 2012, Ma and Yeh \cite{Ma} proved that $T_{G}(1,-1)$ is the number of even-left spanning trees of $G$. To our knowledge, this is the first result for the integer point $(x,y)$ with $xy<0$. Motivated by and based on their works,  we, in this paper, shall first give combinatorial interpretations for $T_{G}(2,-1)$. All undefined terms in this section will be introduced and defined in subsequent sections.
\begin{theorem}\label{thm1}
	Let $G$ be a simple connected graph. Then $T_{G}(2,-1)$ is the number of even-left spanning forests of $G$. 
\end{theorem}
Ma and Yeh \cite{Ma} defined odd $G-$permutations for a simple connected graph $G$ and proved the number of even-left spanning trees is equal to the number of odd $G-$permutations. In this paper, we further define odd $G-$partitionable permutations for a simple connected graph $G$ and prove the following conclusion.
\begin{theorem}\label{thm2}
	Let $G$ be a simple connected graph. Then there is a bijection from $\Upsilon_{G,odd}$ to the set of even-left spanning forests of $G$, where $\Upsilon_{G,odd}$ is the set of odd $G$-partitionable permutations. Hence, $T_{G}(2,-1) = \left| \Upsilon_{G,odd} \right|$.
\end{theorem}

The \textit{alternating permutations} on $[n]=\{1,2,\dots,n\}$ are those permutations with successive pairs of terms alternately in decreasing and increasing order. In general, a permutation $\sigma$ is alternating if $$\sigma(1)>\sigma(2)<\sigma(3)>\sigma(4)<\cdots \sigma(n).$$
Let $S_{n}$ be the set of permutations on $[n]$ and $A_n$ be the set of alternating permutation on $[n]$. Let $a_{n}=\lvert A_n\rvert$. We put $a_0=1$ by convention. 

Let $K_{n}$ be the complete graph with $n$ vertices. Ma and Yeh \cite{Ma} proved that $a_n$ is equal to the number of odd $K_{n+1}-$permutations and also gave a bijection from the set of odd $K_{n+1}-$permutations to $A_{n}$, which indicates that $T_{K_{n+1}}(1,-1)$ is the number of alternating permutations on $[n]$, that is, $T_{K_{n+1}}(1,-1)=a_n$. We have the following analogous conclusion.
\begin{theorem}\label{thm3}
	For any positive integer $n$,  $\left|\Upsilon_{K_{n},odd}\right|  = a_{n+1}$. 
\end{theorem}

Further, we prove the following stronger result than Theorem \ref{thm3}.

\begin{theorem}\label{thm4}
	There is a bijection from $\Upsilon_{K_{n},odd}$ to $A_{n+1}$.
\end{theorem}

There are a lot of research establishing the relation between the value of $T_{K_{n+1}}(x,y)$ at some points and the number of alternating permutations on $[n]$. 
In \cite{Goulden,Kuznetsov}, authors proved that $T_{K_{n+1}}(1,-1)=a_n$. A tree on an ordered vertex set with rooted at minimum vertex is said to be \textit{increasing} if its vertices increase along the paths away from the root. An increasing tree is said \textit{even} if every non-root vertex has an even number of children. Kuznetsov, Pak and Postnikov \cite{Kuznetsov} proved that $a_n$ is equal to the number of even increasing spanning trees of $K_{n+1}$. An increasing tree is called a \textit{0-1-2 increasing tree} if its each vertex has at most $2$ children. Donaghey \cite{Donaghey} gave a bijection from 0-1-2 increasing spanning trees of $K_{n}$ to $A_n$ and hence $a_n$ is also equal to the number of 0-1-2 increasing spanning trees of $K_{n}$. Merino \cite{Merino} proved that $T_{K_n}(2,-1)=T_{K_{n+2}}(1,-1)$ using the method of exponential generating function. A forest is said to be even increasing if its all connected components are even increasing trees. Goodall et al. \cite{Goodall} proved this result again by constructing a bijection from the set of even increasing spanning trees of $K_{n+2}$ to the set of even increasing spanning forests of $K_{n}$. In this paper, we give a new proof for $T_{K_{n}}(2,-1) = a_{n+1}$ in Theorem \ref{thm4} by constructing a new bijection. 

The paper is organized as follows. In Section 2, we shall introduce even-left spanning forests of a graph and prove Theorem \ref{thm1}. In Section 3, we give the definition of odd $G$-partitionable permutations and establish a bijection from them to even-left spanning forests of a graph $G$ to prove Theorem \ref{thm2}.  In Section 4, we prove Theorems \ref{thm3} and \ref{thm4}. 

\section{Even-left spanning forests of a graph}

 For a graph $G$, we denote the set of spanning trees and spanning forests of $G$ by $\mathcal{T}_G$ and $\mathcal{F}_G$, respectively. 
Let $G$ be a simple connected graph with the vertex set $[n]=\{1,2,\dots,n\}$ (can also be any ordered set with cardinality $n$) throughout the paper. Given two disjoint subsets $U$ and $W$ of the vertex set of $G$, we write $E_{G}(U,W)$ for the set of edges joining a vertex in $U$ to a vertex in $W$. 

We first refer to \cite{Ma} to give the definition of even-left spanning trees of $G$. Given a spanning tree $T$ of $G$, we associate $T$ with a sequence$$\pi_{T}=(\pi_T(1),\pi_T(2),\dots,\pi_T(n))$$
on vertices of $G$ by the following algorithm:

\medskip
\noindent\textbf{Algorithm A.}
    \begin{itemize}
        \item \textbf{Step 1.} Let $\pi_T(1)$ be the maximum vertex of $G$.
        \item \textbf{Step 2.} Assume that $\pi_T(1), \pi_T(2), \ldots, \pi_T(i)$ are determined. Let $U_i = \{\pi_T(j) \mid j = 1, 2, \ldots, i\}$. Let $v = \min\{w \in V(G) \setminus U_i \mid E_T(U_i, \{w\}) \neq \emptyset\}$ and set $\pi_T(i + 1) = v$.
    \end{itemize}

For any vertex $v\in V(G)\backslash \{\pi_{T}(1)\}$, there exists a unique vertex $u$ such that $uv \in E(T)$ and $\pi_{T}^{-1}(u)<\pi^{-1}_{T}(v)$. We say $u$ is the predecessor of $v$ and write $u=p_T(v)$. Define $N_{T,v}$ as the set of vertices $w$ such that $wv \in E(G)\backslash E(T)$ and $\pi^{-1}_{T}(p_{T}(v))<\pi^{-1}_{T}(w)<\pi^{-1}_{T}(v)$. Define $N_{T}=\bigcup_{v\in V(G)}\{wv:w\in N_{T,v}\}$.

Given a vertex $v\in V(G)\backslash \{\pi_{T}(1)\}$, for any $w\in V(G)$, say $(w,v)$ is an inversion of the sequence $\pi_{T}$ if $w>v$ and $\pi^{-1}_{T}(w)<\pi^{-1}_{T}(v)$. There exists a unique vertex $r$ such that $(r,v)$ is an inversion of $\pi_T$ and $\pi_{T}(j)<v$ for any $\pi^{-1}_{T}(r)<j<\pi^{-1}_{T}(v)$. We say $r$ is a \textit{$(T;v)$-rightmost inversion vertex}. It is easy to see from Algorithm A that $\pi^{-1}_{T}(r)\le \pi^{-1}_{T}(p_T(v))$. Define $M_{T,v}$ as the set of vertices $s$ such that $sv \in E(G)\backslash E(T)$ and $\pi^{-1}_{T}(r)\le \pi^{-1}_{T}(s)<\pi^{-1}_{T}(p_T(v))$.
\begin{definition}\emph{\cite{Ma}}
Let $G$ be a simple connected graph, and let $T$ be a spanning tree of $G$. A vertex $v\in V(G)\backslash \{\pi_{T}(1)\}$ is said to be $T$-even if $\lvert N_{T,v}\rvert$ is even, and $T$ is said to be even if every vertex except $\pi_{T}(1)$ is $T$-even. A vertex $v\in V(G)\backslash \{\pi_{T}(1)\}$ is said to be $T$-left if $M_{T,v}=\emptyset$, and $T$ is said to be left if every vertex except $\pi_{T}(1)$ is $T$-left. If a spanning tree is both even and left, then it is called an even-left spanning tree.
\end{definition}
Let $\mathscr{T}_{G}$ be the set of even-left spanning trees of $G$. Ma and Yeh \cite{Ma} constructed an involution from $\mathcal{T}_{G} \backslash \mathscr{T}_{G}$ to itself by removing an edge from $E(T)$ and adding an edge from $E(G)\backslash E(T)$. We briefly describe this involution here. For $T\in \mathcal{T}_{G} \backslash \mathscr{T}_{G}$, there are some vertices $w$ such that $|N_{T,w}| \equiv 1\ (\text{mod}\ 2)$ or $M_{T,w} \neq \emptyset$. We select the first vertex $v$ in the sequence $\pi_T$ that satisfies the condition. Suppose $|N_{T,v}| \equiv 1\ (\text{mod}\ 2)$, then let $u$ be the vertex in $N_{T,v}$ such that $\pi^{-1}_{T}(u)\le \pi^{-1}_{T}(w)$ for any $w \in N_{T,v}$. A new spanning tree $T'$ can be obtained by adding the edge $uv$ and removing the edge $p_{T}(v)v$ from $T$. Conversely, as for $T'$, $v$ is also the first vertex in $\pi_{T'}$ that satisfies the above condition, and $|N_{T',v}| \equiv 0\ (\text{mod}\ 2)$, $M_{T',v} \neq \emptyset$. Let $u'$ be the vertex in $M_{T',v}$ such that $\pi_{T'}^{-1}(u')\ge \pi_{T'}^{-1}(w)$ for any $w \in M_{T',v}$. Similarly, a spanning tree $T''$ can be obtained by adding the edge $u'v$ and removing the edge $p_{T'}(v)v$ from $T'$. Then $T=T''$ and this involution is established in this way.

\begin{theorem}\emph{\cite{Ma}}\label{J.Ma1}
	There is an involution $\phi_{G}$ from the set $\mathcal{T}_{G} \backslash \mathscr{T}_{G}$ to itself such that $$\left| \lvert N_T\rvert - \lvert N_{\phi_{G}(T)}\rvert \right| = 1$$ for any $T\in \mathcal{T}_{G} \backslash \mathscr{T}_{G}$.
\end{theorem}
The fact $T_{G}(1,-1) = \left| \mathscr{T}_{G} \right|$ is established by Theorem \ref{J.Ma1} in \cite{Ma}.

We now extend the concept of even-left spanning trees to spanning forests, and then prove that $T_{G}(2,-1)$ is the number of even-left spanning forests of the graph $G$, that is, Theorem \ref{thm1}.

Given a spanning forest $F$ of a graph $G$, suppose that $F$ has connected components $T_1,T_2,\dots,T_{t}$ that are ordered by their maximum vertices. Then $T_i(1\le i\le t)$ can be viewed as a spanning tree of $G_i=G[V(T_i)]$, where $G[V(T_i)]$ denotes the subgraph induced by $V(T_i)$ in $G$. For each $T_i$, we can get sequences $\pi_{T_i}$ on vertices of $G_i$ by Algorithm A. Define the sequence of $F$ from running Algorithm A of its all tree components as $\pi_F=(\pi_{T_1},\pi_{T_2},\dots,\pi_{T_t})$. Define $N_F=\bigcup_{i=1}^{t}N_{T_i}$.

\begin{definition}
	A spanning forest $F$ is even-left if its every connected component $T_{i}(1\le i\le t)$ as a spanning tree of $G_i$ is even-left.
\end{definition}
Let $\mathscr{F}_{G}$ be the set of even-left spanning forests of $G$.

To give a key lemma about the new form of Tutte polynomial, we should introduce some results in \cite{Dimitrije}. 

Let $G$ be a simple graph with $V(G) = [n]$. For any subset $U\subseteq V(G)$ and $v\in U$, define $\text{outdeg}_U(v)$ to be the cardinality of the set $\{vw\in E(G)\,|\,w\notin U\}$. A \emph{$G$-multiparking function} is a function $f : V(G) \to \mathbb{N} \cup \{\infty\}$, such that for every $U \subseteq V(G)$ either (A) $i$ is the vertex of smallest index in $U$, (written as $i = \min(U)$), and $f(i) = \infty$, or (B) there exists a vertex $i \in U$ such that $0 \leq f(v_i) < \text{outdeg}_U(i)$. If $f(v)=\infty$, we say $v$ is a root of $f$.

Kosti\'c{} and Yan \cite{Dimitrije} gave a bijection between $G$-multiparking functions and spanning forests of $G$ by two algorithms, one of which is from $G$-multiparking functions to spanning forests of $G$ and another is the inverse. The two algorithms are determined by a choice function $\gamma$. A forest $F$ of $G$ is a subgraph of $G$ without cycles. A leaf of $F$ is a vertex $v \in V(F)$ with degree $1$ in $F$. Denote the set of leaves of $F$ by $\text{Leaf}(F)$. Let $\Pi$ be the set of all ordered pairs $(F, W)$ such that $F$ is a forest of $G$, and $\emptyset \neq W \subseteq \text{Leaf}(F)$. A \emph{choice function} $\gamma$ is a function from $\Pi$ to $V(G)$ such that $\gamma(F, W) \in W$. Several choice functions were given in Section 3 of \cite{Dimitrije}. We now give one of the two algorithms which gives the inverse of the bijection.

Let $F$ be a spanning forest of $G$. Let $T_1, \dots, T_k$ be the tree components of $F$ with respective minimal vertices $r_1 = 1 < r_2 < \dots < r_k$ that are also roots of them. Fix a choice function $\gamma$. 

\medskip
\noindent\textbf{Algorithm B.}
    \begin{itemize}
        \item \textbf{Step 1.} Define a sequence $\pi = (\pi(1), \pi(2), \dots, \pi(n))$ on the vertices of $G$ as follows. First, $\pi(1) = 1$. Assuming $\pi(1), \pi(2), \dots, \pi(i)$ are determined. If there is no edge of $F$ connecting vertices in $V_i = \{\pi(1), \pi(2), \dots, \pi(i)\}$ to vertices outside $V_i$, let $\pi(i+1)$ be the vertex of smallest index not already in $V_i$. Otherwise, let $W = \{v \notin V_i : v \text{ is adjacent to some vertices in } V_i \text{ in } F\}$, and $F'$ be the forest obtained by restricting $F$ to $V_i \cup W$. Let $\pi(i+1) = \gamma(F', W)$. 
        \item \textbf{Step 2.} Set $f(r_1) = f(r_2) = \dots = f(r_k) = \infty$. For any other vertex $v$, let $s_v$ be the root in the tree containing $v$, and $v, v^p, u_1, \dots, u_t, s_v$ be the unique path from $v$ to $s_v$. Set $f(v)$ to be the cardinality of the set $\{v_j \mid \{v, v_j\} \in E(G), \pi^{-1}(v_j) < \pi^{-1}(v^p)\}$.
    \end{itemize}
\medskip

Then $f$ is a $G$-multiparking function. Notice that the sequence $\pi$ from Step 1 shall process all the vertices of a tree component before processing the vertices of other tree components. In addition, $\pi$ processes tree components by their roots in increasing order.

We give the definition of three types of edges. Let $\pi$ be the sequence obtained in Step 1 of Algorithm B. Let $R^{\gamma}_1(G;F)$ be the set of edges $vw$ of $G$ that both $v$ and $w$ are roots of $F$. Let $R^{\gamma}_2(G;F)$ be the set of edges $vw$ of $G$ that $v$ is a root and $w$ is a non-root of $F$ and $\pi^{-1}(w)<\pi^{-1}(v)$. Let $R^{\gamma}_3(G;F)$ be the set of edges $vw$ of $G$ that $v$ and $w$ are non-roots and $\pi^{-1}(v^p)<\pi^{-1}(w)<\pi^{-1}(v)$, and this implies that $v$ and $w$ must lie in the same tree of $F$. In fact, the union of these edges are $F$-redundant edges introduced in \cite{Dimitrije}. 

Fix a $G$-multiparking function $f$. Given two arbitrary choice functions $\gamma_1,\gamma_2$. From the two bijections, we have the corresponding forests $F_1,F_2$ and sequences $\pi_1,\pi_2$ obtained from Algorithm B. Kosti\'c{} and Yan  took notes on page 86 in \cite{Dimitrije} that the roots of all tree components of $F_1$ or $F_2$ are the roots of $f$, and the two tree components of $F_1$ and $F_2$ respectively that have the same root must have the same vertex set. Moreover, 
$|R^{\gamma_1}_1(G;F_1)|+|R^{\gamma_1}_2(G;F_1)|=|R^{\gamma_2}_1(G;F_2)|+|R^{\gamma_2}_2(G;F_2)|$. These imply that the number of tree components and $|R^{\gamma}_1(G;F)|+|R^{\gamma}_2(G;F)|$ remain unchanged no matter how the choice function $\gamma$ is selected. Hence, we denote the number of tree components of $f$ by $r(f)$ which is also the number of roots of $f$, and denote $|R^{\gamma}_1(G;F)|+|R^{\gamma}_2(G;F)|$ by Rec$(f)$ from \cite{Dimitrije}. 

Kosti\'c{} and Yan \cite{Dimitrije} classified the edges of $G$ by the following proposition.
\begin{proposition}\emph{\cite{Dimitrije}}\label{Kos}
    Given a simple graph $G$, a $G$-multiparking function $f$ and a choice function $\gamma$. $F$ is the corresponding spanning forest of $G$. Then
    \[
        \lvert E(G)\rvert=\sum_{v:f(v)\neq \infty}f(v)+\lvert E(F)\rvert + \text{Rec}(f)+\lvert R^{\gamma}_3(G;F)\rvert.
    \]
\end{proposition}
Kosti\'c{} and Yan gave a form of Tutte polynomial by selecting a special choice function $\gamma_{b}$, that is, 
\[
    T_{G}(1+x,y)=\sum_{F\in \mathcal{F}_G}x^{c(F)-1}y^{\lvert R^{\gamma_b}_{3}(G;F)\vert}
\]
where the sum is over all spanning forests $F$ of $G$. Using Proposition \ref{Kos}, we have 
\[
    T_{G}(1+x,y)=\sum_{f}x^{r(f)-1}y^{\lvert E(G)\rvert-n+r(f)-\text{Rec}(f)-(\sum_{v:f(v)\neq \infty}f(v))}
\]
where the sum is over all $G$-multiparking functions $f$. Note that the exponents on $x$ and $y$ are independent of the choice function $\gamma_{b}$. 
\begin{proposition}
    Given a simple graph $G$ and a choice function $\gamma$. Then
    \[
        T_{G}(1+x,y)=\sum_{F\in \mathcal{F}_G}x^{c(F)-1}y^{\lvert R^{\gamma}_{3}(G;F)\vert}.
    \]
\end{proposition}
We choose another choice function from Section 3 of \cite{Dimitrije}. Given a spanning forest $F$ of a simple connected graph $G$ with $V(G)=[n]$ and a vertex ranking $\sigma\in S_n$, Kosti\'c{} and Yan \cite{Dimitrije} implied that choice functions can be defined by $\gamma_{\sigma}(F,W)=v$, where $v$ is the vertex in $W$ with minimal ranking. We select a vertex ranking $\tau$. In this ranking, the vertex labeled $i$ ranks $n+1-i$. Note that $\gamma_{\tau}(F,W)=v$ where $v$ is the vertex in $W$ with the maximal label.
Then we run Algorithm B with the spanning forest $F$ and the choice function $\gamma_{\tau}$ to obtain the sequence $\pi$ and $R^{\gamma_{\tau}}_3(G;F)$. We relabel the vertex in $G$ by changing the label $i$ to its ranking $n+1-i$ for $1\le i\le n$. Then the sequence $\pi$ is changed to $\pi'$ satisfying $\pi'(i)=n+1-\pi(i)$ for $1\le i\le n$. Now run Algorithm A with all tree components of $F$ and we obtain the sequence $\pi_F$ and $N_F$. Note that $\pi_F=\pi'$ and $\lvert R^{\gamma_{\tau}}_3(G;F)\rvert=\lvert N_F\rvert$. Therefore, we can give the key lemma.

\begin{lemma}\label{Dimi}
	$T_{G}(2,y)=\sum_{F\in \mathcal{F}_G}y^{\lvert N_F\rvert}$.
\end{lemma}
\begin{lemma}\label{lemma1}
	There exists an involution $\psi_{G}$ from $\mathcal{F}_{G} \backslash \mathscr{F}_{G}$ to itself such that $$\left| \lvert N_F\rvert - \lvert N_{\psi_{G}(F)} \rvert \right| = 1$$ for any $F\in \mathcal{F}_{G} \backslash \mathscr{F}_{G}$.
\end{lemma}
\begin{proof}
	Given $F\in \mathcal{F}_{G} \backslash \mathscr{F}_{G}$ with connected components $T_1,T_2,\dots,T_{t}$ that are ordered by their maximum vertices. Suppose $T_{m}$ is the first tree which does not satisfy even-left. Let $\psi_{G}(F) = T_1\cup T_2\cup \cdots \cup T_{m-1}\cup\phi_{G_m}(T_m)\cup\cdots\cup T_t$ where $\phi_{G_m}$ is the involution introduced in Theorem \ref{J.Ma1}. Hence, $\phi_{G_m}(T_m)$ is not even-left and $\phi_{G_m}(\phi_{G_m}(T_m))=T_{m}$ with $\left| \lvert N_{T_m}\rvert - \lvert N_{\phi_{G_m}(T_m)}\rvert \right| = 1$. Then $\phi_{G_m}(T_m)$ is the first tree of $\psi_{G}(F)$ which does not satisfy even-left and $\psi_{G}(\psi_{G}(F))=F$. Note that
	\begin{align*}
		\left| \lvert N_F\rvert - \lvert N_{\psi_{G}(F)}\rvert \right|=&\left|  \sum_{i=1}^{t} \lvert N_{T_i}\rvert -\left (\sum_{i=1}^{m-1}\lvert N_{T_i}\rvert + \lvert N_{\phi_{G_m}(T_m)} \rvert +\sum_{i=m+1}^{t}\lvert N_{T_i}\rvert \right)  \right| \\
		=&\left| \lvert N_{T_m}\rvert - \lvert N_{\phi_{G_m}(T_m)}\rvert \right| =1\\
	\end{align*} 
This completes the proof.
\end{proof}
\begin{figure}[ht]
    \centering
    \begin{tikzpicture}
        \node (4) at (0,1.2) [circle, fill, inner sep=2pt, label=above:4] {};
        \node (3) at (1.2,1.2) [circle, fill, inner sep=2pt, label=above:3] {};
        \node (1) at (0,0) [circle, fill, inner sep=2pt, label=below:1] {};
        \node (2) at (1.2,0) [circle, fill, inner sep=2pt, label=below:2] {};
        \draw (4) -- (3);
        \draw (3) -- (2);
        \draw (2) -- (1);
        \draw (1) -- (4);
        \draw (1) -- (3);
    \end{tikzpicture}
    \caption{ $G$}
    \label{fig:graphG}
\end{figure}

\begin{figure}[ht]
\centering
\begin{tikzpicture}
\foreach \x/\y/\label in {0/0/1, 0.8/0/2, 0/0.8/4, 0.8/0.8/3} {
    \fill (\x,\y) circle (1.2pt);
    \node[below right=-2pt] at (\x,\y) {$\label$};
}
\node[below] at (0.6,-0.4) {$F_1$};
\end{tikzpicture}
\qquad
\begin{tikzpicture}
\foreach \x/\y/\label in {0/0/1, 0.8/0/2, 0/0.8/4, 0.8/0.8/3} {
    \fill (\x,\y) circle (1.2pt);
    \node[below right=-2pt] at (\x,\y) {$\label$};
}
\draw (0,0) -- (0.8,0);
\node[below] at (0.6,-0.4) {$F_2$};
\end{tikzpicture}
\qquad
\begin{tikzpicture}
\foreach \x/\y/\label in {0/0/1, 0.8/0/2, 0/0.8/4, 0.8/0.8/3} {
    \fill (\x,\y) circle (1.2pt);
    \node[below right=-2pt] at (\x,\y) {$\label$};
}
\draw (0.8,0) -- (0.8,0.8);
\node[below] at (0.6,-0.4) {$F_3$};
\end{tikzpicture}
\qquad
\begin{tikzpicture}
\foreach \x/\y/\label in {0/0/1, 0.8/0/2, 0/0.8/4, 0.8/0.8/3} {
    \fill (\x,\y) circle (1.2pt);
    \node[below right=-2pt] at (\x,\y) {$\label$};
}
\draw (0,0.8) -- (0.8,0.8);
\node[below] at (0.6,-0.4) {$F_4$};
\end{tikzpicture}
\qquad
\begin{tikzpicture}
\foreach \x/\y/\label in {0/0/1, 0.8/0/2, 0/0.8/4, 0.8/0.8/3} {
    \fill (\x,\y) circle (1.2pt);
   \node[below right=-2pt] at (\x,\y) {$\label$};
}
\draw (0,0) -- (0,0.8);
\node[below] at (0.6,-0.4) {$F_5$};
\end{tikzpicture}
\qquad
\begin{tikzpicture}
\foreach \x/\y/\label in {0/0/1, 0.8/0/2, 0/0.8/4, 0.8/0.8/3} {
    \fill (\x,\y) circle (1.2pt);
    \node[below right=-2pt] at (\x,\y) {$\label$};
}
\draw (0,0) -- (0.8,0.8);
\node[below] at (0.6,-0.4) {$F_6$};
\end{tikzpicture}
\vspace{1em}

\begin{tikzpicture}
\foreach \x/\y/\label in {0/0/1, 0.8/0/2, 0/0.8/4, 0.8/0.8/3} {
    \fill (\x,\y) circle (1.2pt);
   \node[below right=-2pt] at (\x,\y) {$\label$};
}
\draw (0,0) -- (0.8,0);
\draw (0.8,0) -- (0.8,0.8);
\node[below] at (0.6,-0.4) {$F_7$};
\end{tikzpicture}
\qquad
\begin{tikzpicture}
\foreach \x/\y/\label in {0/0/1, 0.8/0/2, 0/0.8/4, 0.8/0.8/3} {
    \fill (\x,\y) circle (1.2pt);
    \node[below right=-2pt] at (\x,\y) {$\label$};
}
\draw (0,0.8) -- (0.8,0.8);
\draw (0.8,0.8) -- (0.8,0);
\node[below] at (0.6,-0.4) {$F_8$};
\end{tikzpicture}
\qquad
\begin{tikzpicture}
\foreach \x/\y/\label in {0/0/1, 0.8/0/2, 0/0.8/4, 0.8/0.8/3} {
    \fill (\x,\y) circle (1.2pt);
   \node[below right=-2pt] at (\x,\y) {$\label$};
}
\draw (0,0) -- (0.8,0);
\draw (0,0) -- (0,0.8);
\node[below] at (0.6,-0.4) {$F_{9}$};
\end{tikzpicture}
\qquad
\begin{tikzpicture}
\foreach \x/\y/\label in {0/0/1, 0.8/0/2, 0/0.8/4, 0.8/0.8/3} {
    \fill (\x,\y) circle (1.2pt);
    \node[below right=-2pt] at (\x,\y) {$\label$};
}
\draw (0,0) -- (0.8,0);
\draw (0,0.8) -- (0.8,0.8);
\node[below] at (0.6,-0.4) {$F_{10}$};
\end{tikzpicture}
\qquad
\begin{tikzpicture}
\foreach \x/\y/\label in {0/0/1, 0.8/0/2, 0/0.8/4, 0.8/0.8/3} {
    \fill (\x,\y) circle (1.2pt);
    \node[below right=-2pt] at (\x,\y) {$\label$};
}
\draw (0,0) -- (0,0.8);
\draw (0.8,0) -- (0.8,0.8);
\node[below] at (0.6,-0.4) {$F_{11}$};
\end{tikzpicture}
\qquad
\begin{tikzpicture}
\foreach \x/\y/\label in {0/0/1, 0.8/0/2, 0/0.8/4, 0.8/0.8/3} {
    \fill (\x,\y) circle (1.2pt);
    \node[below right=-2pt] at (\x,\y) {$\label$};
}
\draw (0,0) -- (0.8,0.8);
\draw (0,0.8) -- (0.8,0.8);
\node[below] at (0.6,-0.4) {$F_{12}$};
\end{tikzpicture}
\vspace{1em}

\begin{tikzpicture}
\foreach \x/\y/\label in {0/0/1, 0.8/0/2, 0/0.8/4, 0.8/0.8/3} {
    \fill[red] (\x,\y) circle (1.2pt);
   \node[below right=-2pt, text=red] at (\x,\y) {$\label$};
}
\draw[red] (0,0) -- (0,0.8);
\draw[red] (0,0.8) -- (0.8,0.8);
\node[below] at (0.6,-0.4) {$F_{13}$};
\end{tikzpicture}
\qquad
\begin{tikzpicture}
\foreach \x/\y/\label in {0/0/1, 0.8/0/2, 0/0.8/4, 0.8/0.8/3} {
    \fill[green] (\x,\y) circle (1.2pt);
   \node[below right=-2pt, text=green] at (\x,\y) {$\label$};
}
\draw[green] (0,0) -- (0.8,0.8);
\draw[green] (0,0) -- (0,0.8);
\node[below] at (0.6,-0.4) {$F_{14}$};
\end{tikzpicture}
\qquad
\begin{tikzpicture}
\foreach \x/\y/\label in {0/0/1, 0.8/0/2, 0/0.8/4, 0.8/0.8/3} {
    \fill[red] (\x,\y) circle (1.2pt);
    \node[below right=-2pt, text=red] at (\x,\y) {$\label$};
}
\draw[red] (0,0) -- (0.8,0.8);
\draw[red] (0.8,0) -- (0.8,0.8);
\node[below] at (0.6,-0.4) {$F_{15}$};
\end{tikzpicture}
\qquad
\begin{tikzpicture}
\foreach \x/\y/\label in {0/0/1, 0.8/0/2, 0/0.8/4, 0.8/0.8/3} {
    \fill[green] (\x,\y) circle (1.2pt);
    \node[below right=-2pt, text=green] at (\x,\y) {$\label$};
}
\draw[green] (0,0) -- (0.8,0.8);
\draw[green] (0,0) -- (0.8,0);
\node[below] at (0.6,-0.4) {$F_{16}$};
\end{tikzpicture}
\qquad
\begin{tikzpicture}
\foreach \x/\y/\label in {0/0/1, 0.8/0/2, 0/0.8/4, 0.8/0.8/3} {
    \fill[red] (\x,\y) circle (1.2pt);
    \node[below right=-2pt, text=red] at (\x,\y) {$\label$};
}
\draw[red] (0,0) -- (0,0.8);
\draw[red] (0,0.8) -- (0.8,0.8);
\draw[red] (0.8,0) -- (0.8,0.8);
\node[below] at (0.6,-0.4) {$F_{17}$};
\end{tikzpicture}
\qquad
\begin{tikzpicture}
\foreach \x/\y/\label in {0/0/1, 0.8/0/2, 0/0.8/4, 0.8/0.8/3} {
    \fill[green] (\x,\y) circle (1.2pt);
    \node[below right=-2pt, text=green] at (\x,\y) {$\label$};
}
\draw[green] (0,0) -- (0,0.8);
\draw[green] (0,0) -- (0.8,0.8);
\draw[green] (0.8,0) -- (0.8,0.8);
\node[below] at (0.6,-0.4) {$F_{18}$};
\end{tikzpicture}
\vspace{1em}

\begin{tikzpicture}
\foreach \x/\y/\label in {0/0/1, 0.8/0/2, 0/0.8/4, 0.8/0.8/3} {
    \fill[red] (\x,\y) circle (1.2pt);
    \node[below right=-2pt, text=red] at (\x,\y) {$\label$};
}
\draw[red] (0,0) -- (0,0.8);
\draw[red] (0,0) -- (0.8,0.8);
\draw[red] (0,0) -- (0.8,0);
\node[below] at (0.6,-0.4) {$F_{19}$};
\end{tikzpicture}
\qquad
\begin{tikzpicture}
\foreach \x/\y/\label in {0/0/1, 0.8/0/2, 0/0.8/4, 0.8/0.8/3} {
    \fill[green] (\x,\y) circle (1.2pt);
    \node[below right=-2pt, text=green] at (\x,\y) {$\label$};
}
\draw[green] (0,0) -- (0,0.8);
\draw[green] (0,0) -- (0.8,0);
\draw[green] (0.8,0) -- (0.8,0.8);
\node[below] at (0.6,-0.4) {$F_{20}$};
\end{tikzpicture}
\qquad
\begin{tikzpicture}
\foreach \x/\y/\label in {0/0/1, 0.8/0/2, 0/0.8/4, 0.8/0.8/3} {
    \fill[red] (\x,\y) circle (1.2pt);
    \node[below right=-2pt, text=red] at (\x,\y) {$\label$};
}
\draw[red] (0,0.8) -- (0.8,0.8);
\draw[red] (0,0) -- (0.8,0.8);
\draw[red] (0.8,0) -- (0.8,0.8);
\node[below] at (0.6,-0.4) {$F_{21}$};
\end{tikzpicture}
\qquad
\begin{tikzpicture}
\foreach \x/\y/\label in {0/0/1, 0.8/0/2, 0/0.8/4, 0.8/0.8/3} {
    \fill[green] (\x,\y) circle (1.2pt);
   \node[below right=-2pt, text=green] at (\x,\y) {$\label$};
}
\draw[green] (0,0) -- (0.8,0);
\draw[green] (0,0) -- (0.8,0.8);
\draw[green] (0,0.8) -- (0.8,0.8);
\node[below] at (0.6,-0.4) {$F_{22}$};
\end{tikzpicture}
\qquad
\begin{tikzpicture}
\foreach \x/\y/\label in {0/0/1, 0.8/0/2, 0/0.8/4, 0.8/0.8/3} {
    \fill (\x,\y) circle (1.2pt);
    \node[below right=-2pt] at (\x,\y) {$\label$};
}
\draw (0,0) -- (0.8,0);
\draw (0,0.8) -- (0.8,0.8);
\draw (0.8,0) -- (0.8,0.8);
\node[below] at (0.6,-0.4) {$F_{23}$};
\end{tikzpicture}
\qquad
\begin{tikzpicture}
\foreach \x/\y/\label in {0/0/1, 0.8/0/2, 0/0.8/4, 0.8/0.8/3} {
    \fill (\x,\y) circle (1.2pt);
    \node[below right=-2pt] at (\x,\y) {$\label$};
}
\draw (0,0) -- (0,0.8);
\draw (0,0) -- (0.8,0);
\draw (0,0.8) -- (0.8,0.8);
\node[below] at (0.6,-0.4) {$F_{24}$};
\end{tikzpicture}

\caption{All the spanning forests of \( G \) (Black forests are even-left, while red and green forests are not and each adjacent pair is the two sides of the involution $\psi_G$).}
\label{fig:spanning_forests}
\end{figure}
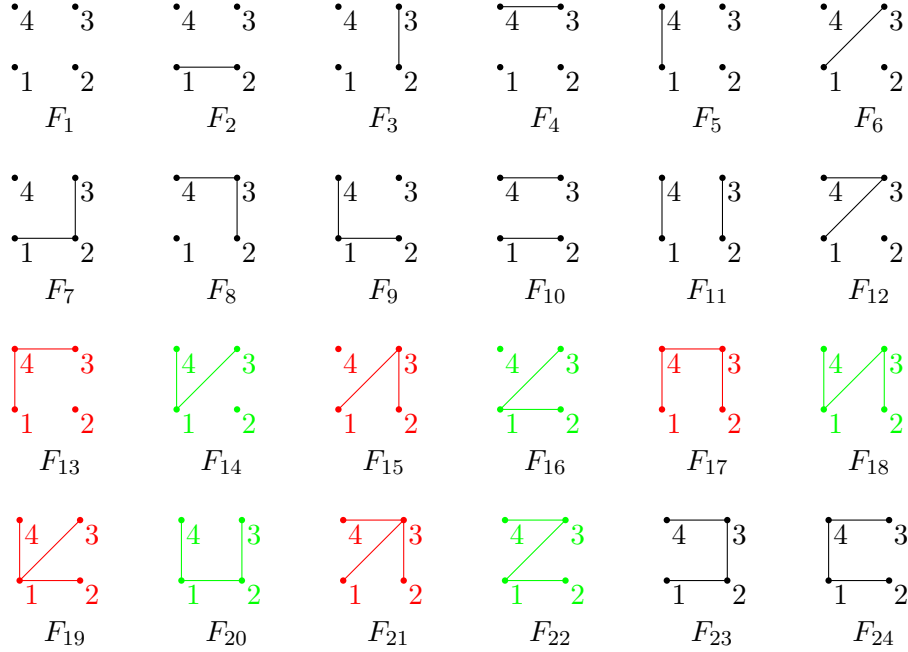
We now present the proof of Theorem \ref{thm1}.

	\textit{Proof of Theorem \ref{thm1}.}  Note that $\left| N_{F} \right|$ is even when $F\in \mathscr{F}_{G}$. Therefore, by Lemmas \ref{Dimi} and \ref{lemma1}, we have
	\begin{align*}
		T_{G}(2,-1) &= \sum_{F\in \mathcal{F}_{G}}(-1)^{\lvert N_F\rvert}\\
		&=\sum_{F\in \mathscr{F}_{G}}(-1)^{\lvert N_F\rvert} + \sum_{F\in \mathcal{F}_{G} \backslash \mathscr{F}_{G}}(-1)^{\lvert N_F\rvert}\\
		&=\sum_{F\in \mathscr{F}_{G}}(-1)^{\lvert N_F\rvert} = \left| \mathscr{F}_{G} \right|.
	\end{align*}
We complete the proof. \qed
\begin{example}
	Let $G$ be shown in Figure \ref{fig:graphG}. We list all the spanning forests of $G$ in Figure \ref{fig:spanning_forests}. Among them, $F_{i}$ with $13\le i\le 22$ are not even-left spanning forests, while all the others are even-left. Moreover, we have $\psi_{G}(F_{13})=F_{14},\ \psi_{G}(F_{15})=F_{16},\ \psi_{G}(F_{17})=F_{18},\ \psi_{G}(F_{19})=F_{20},\ \psi_{G}(F_{21})=F_{22}$ and $T_{G}(2,-1)=14$. These verify Theorem \ref{thm1}.
\end{example}

\section{Odd $G$-partitionable permutations}
When working with permutations of a finite ordered set $(X,<)$, it is often convenient to record a permutation $\sigma$ simply by listing its images in the natural order of the domain. Following the convention of one-line notation for $S_n$, we write permutations of any finite ordered set $X = \{x_1 < \cdots < x_n\}$ as $\sigma(x_1)\sigma(x_2)\cdots\sigma(x_n)$.
This is merely the second line of the two-line notation with the first line understood from the ordering of $X$. For instance, if $X = \{a < b < c\}$ and $\sigma$ sends $a \mapsto c$, $b \mapsto a$, $c \mapsto b$, we write simply $c a b$. When $X = \{1,2,\ldots,n\}$ with the usual order, this reduces to the familiar one-line notation in $S_n$.

We first recall the definition of odd $G$-permutations. Allowing $G$ to be disconnected, given a simple graph $G$ 
and a permutation $$\sigma=a_1a_2\cdots a_n$$ on $V(G)$ such that $a_1$ is the maximum vertex of $G$. We associate $\sigma$ with a spanning forest $F_{\sigma}$ of $G$ by the following algorithm:

\medskip
\noindent\textbf{Algorithm C.}
    \begin{itemize}

	    \item \textbf{Step 1.} Let $E_1 = \emptyset$. 
	    \item \textbf{Step 2.} At time $i \geq 2$, let $j = \max\{k \mid 1 \leq k < i,\ a_k > a_i\}$. Define $\tilde{N}_{\sigma,i}$ as the set of indexes $k$ such that $j \leq k < i$ and $a_k a_i \in E(G)$. If $\tilde{N}_{\sigma,i} = \emptyset$, then let $E_i = E_{i-1}$; otherwise let $E_i = E_{i-1} \cup \{a_m a_i\}$, where $m = \min \tilde{N}_{\sigma,i}$.
    \end{itemize}
\medskip

We can get a set $E_n$ by iterating Step 2 until $i=n$, which induces the spanning forest $F_{\sigma}$. If $F_{\sigma}$ is a spanning tree (in this case, $G$ must be connected) and $\tilde{N}_{\sigma,i}$ is odd for $2\le i\le n$, we say the permutation $\sigma$ to be an \textit{odd $G$-permutation}, and denote the set of odd $G$-permutations of $G$ by $S_{G,odd}$.
\begin{theorem}\emph{\cite{Ma}}\label{J.Ma2}
	$\theta_{G}(\sigma)=F_\sigma$ is a bijection from $S_{G,odd}$ to $\mathscr{T}_{G}$.
\end{theorem}

We now introduce the definition of basic decomposition, which originates from Gessel \cite{Gessel2}. A permutation on a finite ordered set is said to be \emph{basic} if it starts with its maximum value. Consequently, all permutations fed into Algorithm C are basic. Given a permutation $\sigma = a_1a_2\cdots a_n$ on a finite ordered set $M$, we partition $\sigma$ by finding the maximum element one by one. Suppose that $a_{i_1}$ is the maximum element in $M$, and let $M_1=\{a_{i_1},a_{i_1+1},\dots,a_n\}$. Suppose $a_{i_2}$ is the maximum element in $\{a_{1},a_2,\dots,a_{i_1-1}\}$, and let $M_2=\{a_{i_2},a_{i_2+1},\dots,a_{i_1-1}\}$. After a finite number of these steps and assuming that the number is $t$, we can obtain a partition of $M$. Then we let $\sigma_m=a_{i_m}a_{i_m+1}\cdots a_{i_{m-1}-1}$ be a permutation on $M_m$ for $1\le m\le t$ where $i_0-1=n$. Gessel \cite{Gessel2} called the factorization $\sigma=\sigma_t \sigma_{t-1} \cdots \sigma_{1}$ the \emph{basic decomposition} of $\sigma$, and referred to $\sigma_i$ $(1\le i\le t)$ as the \emph{basic components} of $\sigma$. For example, let $\sigma=3268495$ be a permutation on $\{2,3,4,5,6,8,9\}$. Then its basic decomposition is $32\ 6\ 84\ 95$.

Now we shall give the definition of odd $G$-partitionable permutations.
\begin{definition}
	Given a simple connected graph G with vertex $[n]$ and a permutation $\sigma\in S_n$ with the basic decomposition $\sigma_t \sigma_{t-1} \cdots \sigma_1$. Let $G_i$ be the induced subgraph by the vertices in $\sigma_i$ for $1\le i\le t$. The permutation $\sigma$ is an odd $G$-partitionable permutation if $\sigma_i$ is an odd $G_i$-permutation for all $1\le i\le t$.
\end{definition}

	\textit{Proof of Theorem \ref{thm2}.}\quad Given an odd $G$-partitionable permutation $\sigma$ with its basic decomposition $\sigma_t \sigma_{t-1}\cdots \sigma_1$. Let $G_i$ be the induced subgraph of $G$ by the vertices in $\sigma_i$ for $1\le i\le t$. By Theorem \ref{J.Ma2}, we can obtain that $\theta_{G_i}(\sigma_i)$ is an even-left spanning tree of $G_i$ for $1\le i\le t$. Hence $F=\bigcup^{t}_{i=1}\theta_{G_i}(\sigma_i)$ is an even-left spanning forest of $G$. Define a map $\iota$ from $\Upsilon_{G,odd}$ to $\mathscr{F}_{G}$ by $\iota(\sigma)=F$.

    On the one hand, $\iota$ is surjective. For an even-left spanning forest $F$ of $G$, suppose $F$ has connected components $T_1,T_2,\dots,T_{t}$ that are ordered by their maximum vertices. Then $T_i$ is an even-left spanning tree of $G[V(T_i)]$. Let $\sigma_i=\theta^{-1}_{G[V(T_i)]}(T_i)$ for $1\le i\le t$. Then $\sigma_i$ is an odd $G[V(T_i)]$-permutation and notice $\sigma = \sigma_t \sigma_{t-1} \cdots \sigma_1$ is the basic decomposition of $\sigma$. Hence $\sigma\in\Upsilon_{G,odd}$, and we have $\iota(\sigma)=F$.

    On the other hand, $\iota$ is injective. Given two odd $G$-partitionable permutations $\sigma,\tau$ with their basic decompositions $\sigma_r \sigma_{r-1} \cdots \sigma_1$ and $\tau_s \tau_{s-1} \cdots \tau_1$. Suppose $\iota(\sigma)=\iota(\tau)=F$ where $F\in \mathscr{F}_{G}$ has connected components $T_1,T_2,\dots,T_{t}$ that are ordered by their maximum vertices. Hence $r=t=s$ and the induced subgraph of $G$ by the vertex from $\sigma_i$ or $\tau_i$ is $G[V(T_i)]$ for $1\le i\le t$. Moreover, we have $\theta_{G[V(T_i)]}(\sigma_i)=\theta_{G[V(T_i)]}(\tau_i)=T_i$ for $1\le i\le t$. By Theorem \ref{J.Ma2}, $\theta_{G[V(T_i)]}$ is a bijection. Hence for $1\le i\le t$, $\sigma_i=\tau_i$ and $\sigma=\tau$. 
    
    In conclusion, $\iota$ is a bijection from $\Upsilon_{G,odd}$ to $\mathscr{F}_{G}$. Therefore, $T_{G}(2,-1)=\left| \mathscr{F}_{G} \right|=\left| \Upsilon_{G,odd} \right|$. \qed
    
\begin{example}
	Let $G$ be the graph in Figure \ref{fig:graphG}. We list all the permutations $\sigma\in S_4$ in the Table \ref{tab:permutations_S3_full_width} and give the correspondence from odd $G$-partitionable permutations to even-left spanning forests.
\end{example}

\setlength\LTleft{0pt}
\setlength\LTright{0pt}

\begin{longtable}{@{\extracolsep{\fill}}lcccc@{}}
\caption{All the permutations in $S_4$.}\label{tab:permutations_S3_full_width}\\
\hline
 & $\sigma$ & basic decompositions & forests \\
\hline
\endfirsthead  

\multicolumn{5}{c}{Table \thetable\ continued}\\
\hline
 & $\sigma$ & basic decompositions & forests \\
\hline
\endhead  

\hline
\multicolumn{5}{r}{}\\
\endfoot  

\hline
\endlastfoot  

$\sigma_1$ & $1234$ & $1\ 2\ 3\ 4$ & $F_1$ \\
$\sigma_2$ & $1243$ & $1\ 2\ 43$ & $F_4$ \\
$\sigma_3$ & $1324$ & $1\ 32\ 4$ & $F_3$ \\
$\sigma_4$ & $1342$ & $1\ 3\ 42$ & $\backslash$ \\
$\sigma_5$ & $1423$ & $1\ 423$ & $\backslash$ \\
$\sigma_6$ & $1432$ & $1\ 432$ & $F_{8}$ \\
$\sigma_7$ & $2134$ & $21\ 3\ 4$ & $F_2$\\
$\sigma_8$ & $2143$ & $21\ 43$ & $F_{10}$\\
$\sigma_9$ & $2314$ & $2\ 31\ 4$ & $F_{6}$\\
$\sigma_{10}$ & $2341$ & $2\ 3\ 41$ & $F_{5}$\\
$\sigma_{11}$ & $2413$ & $2\ 413$ & $\backslash$\\
$\sigma_{12}$ & $2431$ & $2\ 431$ & $F_{12}$\\
$\sigma_{13}$ & $3124$ & $312\ 4$ & $\backslash$\\
$\sigma_{14}$ & $3142$ & $31\ 42$ & $\backslash$\\
$\sigma_{15}$ & $3214$ & $321\ 4$ & $F_{7}$\\
$\sigma_{16}$ & $3241$ & $32\ 41$ & $F_{11}$\\
$\sigma_{17}$ & $3412$ & $3\ 412$ & $F_{9}$\\
$\sigma_{18}$ & $3421$ & $3\ 421$ & $\backslash$\\
$\sigma_{19}$ & $4123$ & $4123$ & $F_{24}$\\
$\sigma_{20}$ & $4132$ & $4132$ & $\backslash$\\
$\sigma_{21}$ & $4213$ & $4213$ & $\backslash$\\
$\sigma_{22}$ & $4231$ & $4231$ & $\backslash$\\
$\sigma_{23}$ & $4312$ & $4312$ & $\backslash$\\
$\sigma_{24}$ & $4321$ & $4321$ & $F_{23}$\\
\hline
\end{longtable}

\section{A bijection from $\Upsilon_{K_n,odd}$ to $A_{n+1}$}
Consider complete graph $K_n$ with vertex set $[n]$.  We denote $p_{n} = \left| S_{K_{n},odd} \right|$, $q_{n} = \left| \Upsilon_{K_{n},odd}\right|$ and put $p_0=q_0=1$. By direct calculation, we get $p_1=p_2=1$, $q_1=p_3=1$, $q_2=p_4=2$ and $q_3=p_5=5$.

Ma and Yeh \cite{Ma} established the following equation by giving the same recurrence of $p_{n+1}$ as $a_n$, which can also be indirectly deduced from \cite{Goulden,Kuznetsov}.

\begin{theorem}\emph{\cite{Goulden,Kuznetsov,Ma}}\label{J.Ma3}
	$p_{n+1} = a_{n}$ for $n=0,1,2,\dots$.
\end{theorem}
Ma and Yeh \cite{Ma} also constructed a bijection from $S_{K_{n+1},odd}$ to $A_n$.
\begin{theorem}\emph{\cite{Ma}}
	There is a bijection $\Phi_{n}$ from $S_{K_{n+1},odd}$ to $A_n$.
\end{theorem}
If $\sigma\in S_{K_{n+1},odd}$, then $\sigma(1)=n+1$ is fixed to the maximum vertex of $K_{n+1}$. For convenience, when it involves the bijection $\Phi_{n}$, we omit $\sigma(1)$ from $\sigma$.

A classic recurrence formula of $a_n$ is derived from \cite{Goulden} as follows:
\[
    a_{n+2}=\sum_{k=0}^{n}{n\choose k}a_{k+1}a_{n-k},\quad n\ge 0.
\]
Then $p_{n+1}$ has the same recurrence formula as $a_n$ and we have
\begin{equation}\label{eq4}
	p_{n+2}=\sum_{k=0}^{n-1}{n-1\choose k}p_{k+2}p_{n-k},\quad n\ge 1.
\end{equation}

We now prove $q_{n}=a_{n+1}$ by establishing a recurrence formula of $q_{n}$. 

	\textit{Proof of Theorem \ref{thm3}.} Suppose $\sigma\in \Upsilon_{K_{n},odd}$ has the basic decomposition 
    $$\sigma = \sigma_t \sigma_{t-1}\cdots \sigma_2 \sigma_1.$$ 
    Then $\sigma_{1}$ begins with $n$. Suppose there are $m$ $(1\le m\le n)$ numbers in $\sigma_1$. Then $\sigma_{1}$ can be seen as an odd $K_{m}$-permutation and $\sigma' = \sigma_t \sigma_{t-1}\cdots \sigma_2$ can be seen as an odd $K_{n-m}$-partitionable permutation. This suggests that 
	\begin{equation}\label{eq5}
	    q_{n} = \sum_{m=1}^{n}{n-1\choose m-1}q_{n-m}p_{m},\quad n\ge 1.
	\end{equation}
	Hence $q_{n}$ is determined by formula \ref{eq5} with initial condition $q_0=1$. Assert that $q_{n} = p_{n+2}$. In fact, $q_{0} = p_{2} = 1$, and when $q_{n} = p_{n+2}$,  we have, for $n\ge 1$,
	\begin{align*}
		p_{n+2} =& \sum_{m=1}^{n}{n-1\choose m-1}p_{n-m+2}p_{m}\\
		=& \sum_{l=0}^{n-1}{n-1\choose l}p_{n-l+1}p_{l+1} \\
		=& \sum_{k=0}^{n-1}{n-1\choose k}p_{k+2}p_{n-k}.
	\end{align*}
	These are consistent with formula \ref{eq4}.\qed

To give the last bijection from $\Upsilon_{K_n,odd}$ to $A_{n+1}$, we introduce the definitions of relative complement and standardization of a permutation, which are due to \cite{Donaghey} and \cite{Perian}, respectively. Given a permutation $\sigma = a_1a_2\cdots a_n$ and $\{a_i\,|\, 1\le i\le n\}$ need not be $[n]$. Suppose that the rank of $a_i$ in $\{a_i\,|\, 1\le i\le n\}$ is $\sigma_{r}(i)$. We say $R_{\sigma}$ is the \textit{standardization action} of $\sigma$ that $$R_{\sigma}(\sigma) = \sigma_{r}(1) \cdots \sigma_{r}(i)\cdots\sigma_{r}(n)$$
and $R_{\sigma}^{-1}$ is an action to restore the corresponding numbers of $\sigma$. We say $R_{\sigma}(\sigma)\in S_{n}$ to be the \emph{standardization} of $\sigma$. Let $\pi = (n+1-\sigma_{r}(1))\cdots (n+1-\sigma_{r}(i))\cdots (n+1-\sigma_{r}(n))$. We say $R_{\sigma}^{-1}(\pi)$ is the \textit{relative complement} of $\sigma$ and we denote it by $\sigma^{c}$. Let $r_{\sigma}(\sigma) =  (\sigma_{r}(1)-1) \cdots (\sigma_{r}(i)-1) \cdots (\sigma_{r}(n)-1)$ be a permutation on $\{0,1,\dots,n-1\}$ and $r_{\sigma}^{-1}$ be an action to restore the corresponding numbers of $\sigma$.
For example, let $\sigma = 85627$. Then $R_{\sigma}(\sigma) = 52314$, $\sigma^{c} = 27685$ and $r_{\sigma}(\sigma) = 41203$.

We now present the proof of Theorem \ref{thm4}.

	\textit{Proof of Theorem \ref{thm4}.}\quad Let $V(K_{n}) = [n]$. We construct a bijection $\Psi_{n}$ from $\Upsilon_{K_{n},odd}$ to the set of alternating permutations on $\{ 0,1,2\dots,n \}$ by induction $n$. Clearly, $\Psi_1(1)=10$. Then suppose that $\Psi_{m}$ is the bijection from $\Upsilon_{K_{m},odd}$ to the set of alternating permutations on $\{ 0,1,2\dots,m \}$ for $1\le m\le n-1$.
    
	Given a permutation $\sigma \in \Upsilon_{K_{n},odd}$. Suppose the basic decomposition of $\sigma$ is $$\sigma = \sigma_t \sigma_{t-1}\cdots \sigma_2 \sigma_1.$$
	By the definition of odd $G-$partitionable permutations, $R_{\sigma_i}(\sigma_i)$ is an odd $G_{i}$-permutation for $1\le i\le t$ where $G_{i}=K_{k_i}$ and $k_i$ is the number of elements in $\sigma_i$. Suppose there are $n-j+1$ numbers in $\sigma_1$, then we obtain $\sigma(j) = n$ since $\sigma_1$ begins with $n$. Let $\sigma' = \sigma_t \sigma_{t-1}\cdots \sigma_2$ and $\sigma'' = \sigma_{1}\backslash n$. Hence $\sigma=\sigma'n\sigma''$ and $R_{\sigma'}(\sigma')$ is an odd $K_{j-1}$-partitionable permutation. Then $\Psi_{j-1}(R_{\sigma'}(\sigma'))$ is an alternating permutation on $\{0,1,\dots,j-1\}$ and every number except $0$ can be restored by $R_{\sigma'}^{-1}$. We let $s_{\sigma'}$ be an action that $s_{\sigma'}(\sigma')$ is the permutation which from $\Psi_{j-1}(R_{\sigma'}(\sigma'))$ numbers except $0$ are restored by $R^{-1}_{\sigma'}$.
    
    We give the correspondence of $\sigma$ by the following:
	\[
	\Psi_{n}(\sigma) = 
	\begin{cases} 
		n\ 0\ R_{\sigma''}^{-1}\left[ \Phi_{n-1}(R_{\sigma''}(\sigma''))\right], & j = 1; \\[1em]
		\left(\, s_{\sigma'}(\sigma')^{c}\ n\ (R_{\sigma''}^{-1} \left[ \Phi_{n-j}(R_{\sigma''}(\sigma'')) \right])^{c} \,\right)^{c}, & j\ \text{ is odd and }j>1; \\[1em]
		s_{\sigma'}(\sigma')\ n\ (R_{\sigma''}^{-1} \left[ \Phi_{n-j}(R_{\sigma''}(\sigma'')) \right])^{c} \, , & j\ \text{ is even}.
	\end{cases}
	\] 
	It can be easily verified that $\Psi_{n}(\sigma)$ is an alternating permutation on $\{ 0,1,2\dots,n \}$.
    
	Conversely, given an alternating permutation $\tau=\tau(0) \tau(1) \cdots \tau(n)$ on $\{ 0,1,2\dots,n \}$. Note that if $\tau(i) = 0$ and $\tau(j)=n$, then $i$ is odd and $j$ is even. We give the reverse of $\Psi_{n}$ to complete the proof in the following.
    
    \textbf{Case 1.} $j<i$ and $i=1.$
    
	$\tau = n0 \tau(2) \cdots \tau(n)$. Let $\tau'' = \tau(2)\tau(3)\cdots \tau(n)$. Then $\tau''$ is an alternating permutation. Therefore,
	$$\Psi^{-1}_{n}(\tau) = n\ R_{\tau''}^{-1}\left[ \Phi_{n-1}^{-1}(R_{\tau''}(\tau'')) \right].$$

    \textbf{Case 2.} $j<i$ and $i>1.$ 
    
	Notice that $\tau^{c} = (n-\tau(0))(n-\tau(1)) \cdots (n-\tau(n))$ and $\tau^{c}(j)=0,\tau^{c}(i)=n$. Let $\tau' = (\tau^{c}(0) \cdots \tau^{c}(i-1))^{c}$ and $\tau'' = (\tau^{c}(i+1) \cdots \tau^{c}(n))^{c}$. Then $\tau'$ and $\tau''$ are both alternating permutations. Therefore,
	$$\Psi^{-1}_{n}(\tau) = r_{\tau'}^{-1}\left[ \Psi_{i-1}^{-1}(r_{\tau'}(\tau')) \right]\ n\ R_{\tau''}^{-1}\left[ \Phi_{n-i}^{-1}(R_{\tau''}(\tau'')) \right].$$
    
	\textbf{Case 3.} $j>i.$
    
	Let $\tau' = \tau(0)\cdots \tau(j-1)$ and $\tau'' = (\tau(j+1),\dots, \tau(n))^{c}$. Then $\tau'$ and $\tau''$ are both alternating permutations. Therefore, 
	$$\Psi^{-1}_{n}(\tau) = r_{\tau'}^{-1}\left[ \Psi_{j-1}^{-1}(r_{\tau'}(\tau')) \right]\ n\ R_{\tau''}^{-1}\left[ \Phi_{n-j}^{-1}(R_{\tau''}(\tau'')) \right].$$
    This completes the proof.\qed

\begin{example}
	We give the bijections of complete graph $K_1,K_2,K_3,K_4$ in Table \ref{Bijections}.
\end{example}

\begin{table}[ht]
\centering
\caption{Bijections of $K_1,K_2,K_3,K_4$}
\label{Bijections}
\begin{tabular*}{\linewidth}{@{\extracolsep{\fill}}lcccc@{}}
\hline
 & $\Upsilon_{K_{n},odd}$ & decompositions & alternating permutations  \\
\hline
$K_1$ & $1$ & $1$ & $10$ \\
$K_2$ & $12$ & $1\ 2$ & $102$ \\
      & $21$ & $21$ & $201$ \\
$K_3$ & $123$ & $1\ 2\ 3$ & $2130$ \\
      & $132$ & $1\ 32$ & $1032$ \\
      & $213$ & $21\ 3$ & $3120$ \\
      & $231$ & $2\ 31$ & $2031$\\
      & $321$ & $321$ & $3021$\\
$K_4$ & $1234$ & $1\ 2\ 3\ 4$ & $21304$\\
      & $1243$ & $1\ 2\ 43$ & $32401$\\
      & $1324$ & $1\ 32\ 4$ & $10324$\\
      & $1342$ & $1\ 3\ 42$ & $31402$\\
      & $1432$ & $1\ 432$ & $10423$\\
      & $2134$ & $21\ 3\ 4$ & $31204$\\
      & $2143$ & $21\ 43$ & $42301$\\
      & $2314$ & $2\ 31\ 4$ & $20314$\\
      & $2341$ & $2\ 3\ 41$ & $21403$\\
      & $2431$ & $2\ 431$ & $20413$\\
      & $3142$ & $31\ 42$ & $41302$\\
      & $3214$ & $321\ 4$ & $30214$\\
      & $3241$ & $32\ 41$ & $41203$\\
      & $3421$ & $3\ 421$ & $30412$\\
	   & $4321$ & $4321$ & $40312$\\
      & $4213$ & $4213$ & $40213$\\
\hline
\end{tabular*}
\end{table}

\noindent\emph{Remark:} As we know, the number of even-left spanning trees of $K_{n+2}$, 0-1-2 increasing spanning trees of $K_{n+1}$, 0-1-2 increasing spanning forests of $K_{n}$ at most two components, even increasing spanning trees of $K_{n+2}$ and even increasing spanning forests of $K_{n}$ are all equal to $a_{n+1}$. From Theorem \ref{thm4}, the bijections between even-left spanning forests of $K_{n}$ and those above trees or forests can be given explicitly, indirectly through the bijections between them and alternating permutations.
\newpage
\section*{Acknowledgements}

This work is supported by National Natural Science Foundation of China
(Nos. 12571379, 12571366) and Scientific Research Start-Up Foundation of Jimei University (No. ZQ2024116).

\section*{Declarations}
\noindent
{ \bf Conflict of interest}
The authors declare that they have no conflict of interest.

\section*{References}
\begingroup
\renewcommand{\section}[2]{}%

\endgroup
\end{document}